 \documentclass[11pt]{article}
\usepackage{amsmath,amsthm,amsfonts,epsfig}
\usepackage{amssymb}
\usepackage{enumerate}
\usepackage[mathscr]{euscript}
\usepackage{tikz}
\usepackage{amssymb}
\usepackage{amsmath}
\usepackage{amsthm}
\usepackage{mathrsfs}
\usepackage{geometry}

\newcommand{\N}{\mathbb{N}}

\newcommand{\R}{\mathbb{R}}
\newcommand{\C}{\mathbb{C}}

\newcommand{\rex}{\rho_\varepsilon^X}
\newcommand{\rey}{\rho_\varepsilon^Y}
\newcommand{\eps}{\varepsilon}
\newcommand{\de}{d_\varepsilon}

\DeclareMathOperator{\dist}{dist}

\newtheorem{thm}{Theorem}[section]
\newtheorem{lemma}[thm]{Lemma}

\theoremstyle{definition}
\newtheorem{defn}[thm]{Definition}

\newtheorem{rmk}[thm]{Remark}

\title{Rough isometry between Gromov hyperbolic spaces and uniformization}
\author{Jeff Lindquist and Nageswari Shanmugalingam
\footnote{The second author was partially supported by grant DMS~\#1800161 from NSF (U.S.A.) Part of the work for
this project was done while the second author was visiting IMPAN; she thanks that institution and the Simons 
Foundation grant 346300 to IMPAN and the matching 2015-2019 Polish MNiSW fund for their kind hospitality.
The authors also thank the referees for suggestions that helped improve the exposition.}}

\begin{document}
\maketitle

\begin{abstract}
In this note we show that given two complete geodesic
Gromov hyperbolic spaces that are roughly isometric and an arbitrary
$\eps>0$ (not necessarily small), 
either the uniformization of both spaces with parameter $\eps$ results in uniform
domains, or else neither uniformized space is a uniform domain. The terminology of 
``uniformization" is from~\cite{BHK}, where it is shown that the uniformization,
with parameter $\eps>0$, of a complete
geodesic Gromov hyperbolic space results in a uniform domain provided $\eps$ is small enough.
\end{abstract}

\section{Introduction}

Uniform domains play a special role in the study of planar quasiconformal mappings (see
for example~\cite{MarSar} where the concept of uniform domains was first introduced, \cite{Mar, GeO,
BKR, H, KL}) 
and in potential theory (see for example~\cite{KP, KT, LLMS, A1, A2, HK, BSh}). 
The notion of uniform domains does not require the underlying space to
be Euclidean or smooth, and so has a natural extension to general
metric spaces, see Definition~\ref{def:uniform} below. On the other hand, the notion of curvature, as
defined in Riemannian geometry, is a second order calculus notion and so does not easily lend itself
to the setting of more general metric spaces. Instead, in that non-smooth setting, the role of 
negative curvature is played by two possible alternatives, Alexandrov curvature and Gromov
hyperbolicity, see the discussions in~\cite{BH, BuSch, CDP, GH}.
Gromov hyperbolic spaces were first defined in~\cite{Gr} in the 
context of studying hyperbolic groups. 

The work \cite{BHK} demonstrates a strong connection between Gromov hyperbolic spaces and 
uniform domains. It was shown there that uniform domains in metric spaces, equipped with the
quasihyperbolic metric $k$ (see~\eqref{eq:quasihyp-defn}) are necessarily  Gromov hyperbolic spaces.
Conversely, given a geodesic 
Gromov hyperbolic space $X$, there is a positive number $\eps_0$ such that whenever
$0<\eps\le \eps_0$, the uniformization $X_\eps$ of $X$ corresponding to the parameter $\eps$ is
a uniform domain.

It is not difficult to see that if $X$ and $Y$ are two complete geodesic spaces with $Y$ a Gromov 
hyperbolic space, and if there is a rough isometry $\Phi:Y\to X$ as in Definition~\ref{def:rough-isom},
then $X$ is also Gromov hyperbolic; that is, Gromov hyperbolicity is a large scale property and is
not destroyed by small-scale perturbations. Therefore it is natural to ask whether the allowable range
of uniformization parameters is preserved by rough isometries. This is the goal of this current
note. In particular, we show that if $X$ and $Y$ are Gromov hyperbolic and $\Phi:Y\to X$ is a
rough isometry, and if $\eps>0$ is such that $X_\eps$ is a uniform domain, then $Y_\eps$ is also
a uniform domain, see Theorem~\ref{thm:main}.
In~\cite{BBS} it was shown that if a Gromov hyperbolic space $X$ is uniformized to a uniform
domain $X_\eps$ (for sufficiently small $\eps>0$), and the subsequent boundary
$Z:=\partial X_\eps$ has a hyperbolic filling $Y$ with appropriate scaling parameters,
then $Y$ is roughly isometric to $X$. It follows from our results then that $Y_\eps$ is also
a uniform domain (since we know that $X_\eps$ is). It is not difficult to see that
$\partial Y_\eps$ is isometric to $\partial X_\eps$, and hence our result 
ties the potential theoretic
properties of $\partial X_\eps$ to those of $Y_\eps$, even though $X_\eps$ itself could be
ill-connected from the point of view of potential theory. It was shown in~\cite{BBS} that
if $Y_\eps$ is a uniform domain, then
$Y_\eps$ has a suitable measure with respect to which $Y_\eps$ is doubling and supports a 
$1$-Poincar\'e inequality. 

If $Z$ is a doubling metric measure space supporting a $p$-Poincar\'e inequality, then the
correct setting for potential theory on $Z$ is the so-called Newton-Sobolev class of functions,
see for example~\cite{BBbook}. When $Z$ does not support such a Poincar\'e inequality, 
for example if $Z$ does not have sufficiently many rectifiable curves, then the Newton-Sobolev
class is the wrong class for potential theory on $Z$; in this case, the more appropriate 
function class is a Besov class of functions on $Z$. 
See for example~\cite{GKS, BBS, BBS2} for more on Besov classes.
There are Gromov hyperbolic spaces $X$
for which $X_\eps$ is a uniform domain but $\partial X_\eps$ may not even be connected; hence
the potential theory on $\partial X_\eps$ should be via Besov classes. As Besov energies are
non-local, their properties are not well understood. For example, what metric properties of
subsets of $\partial X_\eps$ guarantee that the Besov capacity of the set is null? 
Should $Y$ be another Gromov hyperbolic metric space equipped with a uniformly locally
doubling measure supporting a uniformly local Poincar\'e inequality, and $X$ is roughly isometric
to $Y$, then we know from our main theorem, Theorem~\ref{thm:main}, that $Y_\eps$ is also
a uniform space. It then follows from the results in~\cite{BBS2} that the induced
measure on $Y_\eps$ is doubling and supports a Poincar\'e inequality; the results of~\cite{BBS2}
show that the trace of the Newton-Sobolev class of functions on $Y_\eps$ \emph{is} a Besov
class on $\partial Y_\eps$, and that a subset of $\partial Y_\eps$ is null for this Besov class
if and only if it is null for the Newton-Sobolev class. Null sets for Newton-Sobolev classes
are reasonably well understood, and this understanding translates to a reasonable understanding
of Besov-null sets in $\partial Y_\eps$. Since $X$ is roughly isometric to $Y$, we must have
that $\partial X_\eps$ is biLipschitz equivalent to $\partial Y_\eps$, and so such an
understanding can be imported to $\partial X_\eps$ as well. See~\cite{BoPa} for connections between
Besov classes on $\partial X_\eps$ and the algebraic structure of $X$ for certain geometric classes of
hyperbolic spaces $X$.

Observe that by the results of~\cite{BHK}, $Y_\eps$ is a 
uniform domain if $\eps$ is small enough, but here we do not require smallness of $\eps$.
The key reason in~\cite{BHK} for requiring $\eps$ be sufficiently small is that for small
enough $\eps$ a Gehring-Hayman property holds for hyperbolic geodesics. Since we do not
assume $\eps$ to be small, we cannot rely on this property; instead, our proof uses
the technique of discretization of paths.

The next section is devoted to providing the relevant definitions. The
first part of the third section develops the tools necessary to prove our main theorem,
Theorem~\ref{thm:main}, and the proof of that theorem is given in the last part of that section.
The last section is devoted to showing that replacing rough isometries with rough quasiisometries
makes the conclusion of Theorem~\ref{thm:main} false.

We adopt the convention that $Q_1\gtrsim Q_2$ if there is a constant $C>0$ such that
$C\, Q_1\ge Q_2$. We say that $Q_1\lesssim Q_2$ if $Q_2\gtrsim Q_1$, and we say that $Q_1\simeq Q_2$ if 
$Q_1\gtrsim Q_2$ and $Q_1\lesssim Q_2$. We say that $Q_1\simeq Q_2$ with comparison constant $C>0$
if 
\[
\frac{1}{C}\, Q_1\le Q_2\le C\, Q_1.
\]

\section{Background}

In this paper we only consider complete, locally compact, geodesic metric spaces.
We provide the relevant definitions of the notions used in this note. In what follows,
given a metric space $(Z,d)$, $z\in Z$ and $r>0$, we set 
$B(z,r):=\{x\in Z\, :\, d(x,z)<r\}$ and $\overline{B}(z,r):=\{x\in Z\, :\, d(x,z)\le r\}$.

\begin{defn}\label{def:GromovHyp}
A complete unbounded locally compact
geodesic metric space $(Z,d)$ is said to be \emph{Gromov hyperbolic} if there exists 
$\delta\ge 0$ such that whenever $x,y,z\in Z$ and $[x,y], [y,z], [z,x]$ are geodesic paths in $Z$
with end points $x,y$, end points $y,z$, and end points $z,x$, respectively, then
\[
[x,y]\subset \bigcup_{w\in[y,z]\cup[z,x]}\overline{B}(w,\delta).
\]
Here, if $\delta=0$, we interpret $\overline{B}(w,\delta)$ to be the set $\{w\}$.
\end{defn}

The above definition of Gromov hyperbolic space is from~\cite{BHK}, but readers might want to
keep in mind that there are alternate definitions of Gromov hyperbolic spaces in literature
that do not require the metric space to be geodesic or locally compact, see for example
\cite[Definition~1.2.2, Proposition~2.1.2, Proposition~2.1.3]{BuSch}.

\begin{defn}\label{def:RoughStarlike}
We say that a Gromov hyperbolic space $(Z,d)$ is \emph{roughly starlike} (or, $M$-roughly starlike)
if there exists $M\ge 0$ and $z_0\in Z$ such that
for all $z\in Z$ there is a geodesic ray $\gamma:[0,\infty)\to Z$ with $\gamma(0)=z_0$ 
and $t_0\in[0,\infty)$ such that $d(z,\gamma(t_0))\le M$.
\end{defn}

Trees with each vertex of degree at least $2$
are Gromov hyperbolic with $\delta=0$ and are roughly starlike with $M=0$. Uniform domains,
equipped with the quasihyperbolic metric, are necessarily Gromov hyperbolic and roughly starlike, 
see the discussion in~\cite[Chapter~3]{BHK}. The example 
\[
X=[0,\infty)\cup\bigcup_{n\in\N}[n,n+\sqrt{-1}\, n]\subset\C,
\]
equipped with the length metric induced by the Euclidean metric shows that there are
Gromov hyperbolic spaces that are not roughly starlike.

Following~\cite{BHK}, for each $\eps>0$ we consider uniformization of Gromov hyperbolic spaces
with parameter $\eps$.

\begin{defn}\label{def:uniformization}
Let $(Z,d)$ be a Gromov hyperbolic space, $z_0\in Z$, and $\eps>0$. We consider the ``density" function
$\rho_\eps^Z:Z\to(0,1]$ given by 
\[
\rho_\eps^Z(z):=e^{-\eps d(z,z_0)}.
\]
This density function induces a metric on $Z$, given by
\[
d_\eps(z_1,z_2)=\inf_\gamma\int_\gamma\rho_\eps^Z\, ds,
\]
for $z_1,z_2\in Z$, where the infimum is over all rectifiable paths $\gamma$ in $Z$ with end points
$z_1$ and $z_2$ and the integral is taken with respect to the arc-length measure $ds$. 
We denote this induced metric space $(Z,d_\eps)$ by $Z_\eps$. Furthermore, for a locally
rectifiable path $\gamma$ in 
$Z$, we denote by $\ell(\gamma)$ its length with respect to the metric $d$,
and by $\ell_\eps(\gamma)$ its length with respect to the metric $d_\eps$. It is not difficult to
see that
\[
\ell_\eps(\gamma)=\int_\gamma \rho_\eps^Z\, ds.
\]
\end{defn}

The above construction of uniformization is from~\cite[Chapter~4]{BHK}. 
As mentioned above, from~\cite{BHK} we know that if $Z$ is Gromov
hyperbolic and $\eps\le \eps_0=\eps_0(\delta)$, then $Z_\eps$ is a uniform domain, that is, it satisfies
the following definition.

\begin{defn}\label{def:uniform}
Let $Z$ be a locally complete, non-complete metric space, and set $\partial Z:=\overline{Z}\setminus Z$. We say that
$Z$ is a \emph{uniform domain} (or a uniform space) if there is a constant $\lambda\ge 1$ such that for each pair of points
$x,y\in Z$ there is a rectifiable curve $\gamma$ in $Z$ with end points $x$ and $y$ satisfying
\begin{enumerate}
\item $\ell(\gamma)\le \lambda\, d(x,y)$, that is, $\gamma$ is $\lambda$-quasiconvex,
\item for each $z\in\gamma$,
\[
 \delta_Z(z):=\text{dist}(z,\partial Z)\ge \lambda^{-1}\min\{\ell(\gamma(x,z)),\ell(\gamma(z,y))\}.
\]
\end{enumerate}
Here $\gamma(x,z)$ is any of the subcurves of $\gamma$ with end points $x,z$.
The number $\lambda$ is called a uniformity constant of $Z$, and a curve $\gamma$ satisfying
the two listed conditions above is said to be a uniform curve or a $\lambda$-uniform curve.
\end{defn}

From~\cite{MarSar, GeO, BHK}, there is a natural deformation
of the metric on a uniform domain $(Z,d)$, called the quasihyperbolic metric.

\begin{defn}\label{def:quasihyperbolic}
Given a locally compact, non-complete metric space $(Z,d)$, the \emph{quasihyperbolic metric} $k$ on $Z$
is given by
\begin{equation}\label{eq:quasihyp-defn}
k(x,y)=\inf_\gamma\int_\gamma\frac{1}{\delta_Z(\gamma(t))}\, ds(t)
\end{equation}
when $x,y\in Z$. Here the infimum is over all rectifiable curves $\gamma$ in $Z$ with end points
$x$ and $y$, and $ds$ is the arc-length measure.
\end{defn}

We assume from now on that $(X, d)$ and $(Y, d)$ are Gromov hyperbolic spaces.

\begin{defn}\label{def:rough-isom}
A map $\Phi \colon Y \to X$ is a \emph{$\tau$-rough isometry} if
\[
d(y_1,y_2) - \tau \leq d(\Phi(y_1),\Phi(y_2)) \leq d(y_1,y_2) + \tau
\]
for all $y_1,y_2 \in Y$ 
and $\Phi(Y)$ is \emph{$\tau$-dense} in $X$, that is,  for each $x\in X$ there is some $y_x\in Y$ such that
$d(x,\Phi(y_x))\le \tau$. 
\end{defn}

Note that we do not require $\Phi$ to be continuous, and we do not require it to be injective or surjective.

\begin{lemma}\label{lem:inverse-rough-isom}
Given a $\tau$-rough isometry $\Phi:Y\to X$, there exists 
a $3\tau$-rough isometry $\Phi^{-1}:X\to Y$ such that
for all $y\in Y$ and $x\in X$ we have
\[
d(y,\Phi^{-1}(\Phi(y)))\le 2\tau, \qquad d(x,\Phi(\Phi^{-1}(x)))\le \tau.
\]
\end{lemma}

This seems to be well-known (see for example~\cite{BS}, \cite[Proposition~3.2]{Bow}), 
but as we were not able to find 
a published proof of this fact, we provide the proof here for the convenience of the reader.

\begin{proof}
We first construct $\Phi^{-1}:X\to Y$ as follows. Given $x\in X$, by the fact that every point in
$X$ is within a distance $\tau$ of $\Phi(Y)$, we can find a point $y_x\in Y$ such that
$d(\Phi(y_x),x)\le \tau$. We choose one such $y_x$ and set $\Phi^{-1}(x)=y_x$. 
Note that 
\[
d(\Phi(\Phi^{-1}(x)),x)=d(\Phi(y_x),x)\le \tau.
\] 
Moreover, for $y\in Y$, with the choice of $x=\Phi(y)$, we have the point 
$y_{\Phi(y)}$ as a point in $Y$ that $\Phi^{-1}$ maps $x$ to. Then
$d(\Phi(y_{\Phi(y)}),x)\le \tau$, and so
\[
d(\Phi^{-1}(\Phi(y)),y)=d(y_{\Phi(y)},y)\le \tau+d(\Phi(y_{\Phi(y)}),\Phi(y))\le 2\tau.
\]
For $x,x'\in X$, we have
\begin{align*}
 d(\Phi^{-1}(x),\Phi^{-1}(x'))=d(y_x,y_{x'})&\le \tau+d(\Phi(y_x),\Phi(y_{x'}))\\
  &\le \tau+d(\Phi(y_x),x)+d(x,x')+d(x',\Phi(y_{x'}))\\
  &\le 3\tau+d(x,x').
\end{align*}
Furthermore,
\begin{align*}
 d(\Phi^{-1}(x),\Phi^{-1}(x'))=d(y_x,y_{x'})&\ge d(\Phi(y_x),\Phi(y_{x'}))-\tau\\
  &\ge -d(\Phi(y_x),x)+d(x,x')-d(x',\Phi(y_{x'}))-\tau\\
  &\ge d(x,x')-3\tau.
\end{align*}
Finally, given $y\in Y$, we set $x=\Phi(y)$ and note from the first part of the argument that
\[
d(y,\Phi^{-1}(x))=d(y,\Phi^{-1}(\Phi(y)))\le 2\tau.
\]
This concludes the  proof.
\end{proof}

\begin{rmk}\label{rmk:roughed up}
Note that if $\Phi$ is a $\tau$-rough isometry, then it is also a $3\tau$-rough isometry. Hence,
by replacing $\tau$ with $3\tau$ if necessary, we will assume in the rest of the paper that
both $\Phi$ and $\Phi^{-1}$ are $\tau$-rough isometries with
\[
d(y,\Phi^{-1}(\Phi(y)))\le \tau \quad \text{ and } \quad d(x,\Phi(\Phi^{-1}(x)))\le \tau.
\]
\end{rmk}

\begin{rmk}
Suppose that $X$ and $Y$ are two geodesic
metric spaces and $\Phi:Y\to X$ is a $\tau$-rough isometry.
From~\cite[Proposition~1.22]{BH} we know that if $X$ is $\delta$-Gromov hyperbolic, then
the Gromov product $(x\vert y)_{y_0}$, $x,y\in Y$, satisfies the so-called $6\delta$-inequality:
\[
(x\vert y)_{y_0}\ge \min\{(x\vert z)_{y_0},(z\vert y)_{y_0}\}-6\delta,
\]
where the Gromov product is defined by
\[
(x\vert y)_{y_0}:=\frac12[d(x,y_0)+d(y,y_0)-d(x,y)].
\]
Moreover, if $Y$ is a geodesic space
and satisfies the above $6\delta$-inequality, then $Y$ is $36\delta$-Gromov hyperbolic.
From the above it follows immediately that if $Y$ is $\delta$-Gromov hyperbolic, then
$X$ is $6(3\tau+6\delta)$-Gromov hyperbolic. 

If $Y$ is both $\delta$-Gromov hyperbolic and $M$-roughly starlike, then $X$ is 
$M^\prime(\delta,\tau, M)$-roughly starlike. To see this, note that if $x\in X$, then
setting $y=\Phi^{-1}(x)$, by the rough starlikeness of $Y$ there is a geodesic ray
$\gamma:[0,\infty)\to Y$ with $\gamma(0)=y_0$ and some $t_0\ge 0$ such that
$d(\gamma(t_0),y)\le M$. For $k=0,1,\cdots$ let $b_k=\Phi(\gamma(k(1+\tau)))$ 
and $w=\Phi(\gamma(t_0))$.
Then $d(x,w)\le d(y,\gamma(t_0))+\tau\le M+\tau$, and 
$d(w,b_i)\le d(\gamma(t_0),\gamma(i(1+\tau)))+\tau\le 1+2\tau$ for some $i\in\N$. 
By the geodesic stability result~\cite[Theorem~1.3.2 of page~5]{BuSch}, 
together with an invocation of the Arzel\`a-Ascoli theorem there is a 
positive number $h(\tau,\delta)$ and a geodesic
ray $\beta:[0,\infty)\to X$ with $\beta(0)=x_0$ and $s_0\ge 0$ such that 
$d(b_i,\beta(s_0))\le h(\tau,\delta)$. Combining these together, we get
\[
d(x,\beta(s_0))\le M+\tau+1+2\tau+h(\tau,\delta)=1+M+3\tau+h(\tau,\delta),
\]
that is, $X$ is $M^\prime(\delta,\tau, M)$-roughly starlike with
\[
M^\prime(\delta,\tau, M)=1+M+3\tau+h(\tau,\delta).
\]
Interestingly, the geodesic stability property mentioned above also 
characterizes Gromov hyperbolicity, see~\cite{Bo}. The result~\cite[Proposition~3.1]{Bo}
together with the fact that the path in $X$ obtained by concatenating the geodesic
segments connecting $b_k,b_{k+1}$, $k=0,1,\cdots$ is a $(\lambda,\tau\lambda)$-chord-arc
curve in the sense defined in~\cite[Page~295]{Bo} gives a more explicit estimate for
$h(\tau,\delta)$ than that found in~\cite{BuSch}. Here,
\[
\lambda=\frac{1+2\tau}{1+\tau}.
\]
\end{rmk}

There is a more general notion of Gromov hyperbolicity that does not require the space
to be a geodesic space, and this notion is given with respect to the Gromov product;
see for example~\cite{Gr, GH}. Since
what we do with path integrals requires our space to be a geodesic space, we do not
consider the Gromov product definition of hyperbolicity.

\begin{rmk}
The density $\rho_\eps^Z$ as considered in Definition~\ref{def:uniformization} is an example of 
a large class of densities, called \emph{conformal densities}, used to deform metrics on 
a given metric space, see for example~\cite{KL, BKR}. 
A positive continuous function $\rho$ on a metric space $Z$ is
a \emph{Harnack conformal density} if there is a constant $A\ge 1$ such that 
whenever $x, y \in X$ with $d(x,y)\le 1$, we have 
\begin{equation}\label{ratio comp}
\frac{1}{A} \leq \frac{\rho(x)}{\rho(y)} \leq A.
\end{equation}
The nomenclature is justified by the fact that if $\rho$ is a conformal density on $(Z,d)$ and
the metric on $Z$ is modified to a new metric $d_\rho$
according to the scheme given in Definition~\ref{def:uniformization}
with $\rho$ playing the role of $\rho_\eps^Z$, then the natural identity map
$\text{Id}:(Z,d)\to (Z,d_\rho)$ is a (metrically) $1$-quasiconformal map. The usage of ``Harnack"
in the above nomenclature echoes the Harnack property of positive harmonic functions.
\end{rmk}

We fix two distinguished points $x_0\in X$ and $y_0\in Y$.
We are concerned with the two densities 
\[
\rex(x) = e^{-\eps d(x_0, x)} \ \text{ and } \ \rey(y) = e^{-\eps d(y_0, y)}.  
\]
We denote by $X_\eps$ and $Y_\eps$ the 
$\eps$-uniformizations of $X$ and $Y$.

\begin{rmk}\label{mult dist rmk}
Given a conformal density $\rho$ on $Z$ as in~\eqref{ratio comp}, and $Z$ a geodesic space,
we see that whenever $K \in \N$ and $x, y \in X$ such that $d(x,y) \leq K$, then 
\[
\frac{1}{A^K} \leq \frac{\rho(x)}{\rho(y)} \leq A^K.
\]
Note that by the triangle inequality,
\[
\frac{\rex(x)}{\rex(y)}=e^{-\eps[d(x,x_0)-d(y,x_0)]}\ge e^{-\eps d(x,y)}\ge e^{-\eps}
\]
when $d(x,y)\le 1$. Similarly, we get
\[
\frac{\rex(x)}{\rex(y)}=e^{-\eps[d(x,x_0)-d(y,x_0)]}\le e^{\eps d(x,y)}\le e^\eps.
\]
Thus both $\rex$ and $\rey$ satisfy~\eqref{ratio comp} with $A=e^\eps$.
\end{rmk}

As described above, a given roughly starlike Gromov hyperbolic space can be uniformized
and then the resulting space can be equipped with its quasihyperbolic metric 
(see~\eqref{eq:quasihyp-defn} above for the definition of quasihyperbolic metric). The outcome may
not be isometric to the original Gromov hyperbolic space, but as the next lemma shows, it is
biLipschitz equivalent.

\begin{lemma}\label{lem-heartburn}
Let $(X,d)$ be a roughly starlike Gromov hyperbolic space and $\eps>0$.
Then $(X_\eps,k)$ is biLipschitz equivalent to $(X,d)$.
\end{lemma}

In the above lemma, $k$ is the quasihyperbolic metric with respect to the uniformized space $X_\eps$. Note that
we do not assume any condition on $\eps$ apart from that it is positive. The above lemma was proved
in~\cite[Proposition~4.37]{BHK} for the setting where $\eps\le \eps_0$. For the convenience of the 
reader, we provide a short proof of Lemma~\ref{lem-heartburn} here.

\begin{proof}
Let $\delta_\eps$ be the distance to the boundary $\partial X_\eps:=\overline{X_\eps}\setminus X_\eps$.
Recall that we have a distinguished point $x_0\in X$ in the definition of $X_\eps$.
Note that the quasihyperbolic distance $k$ is given by
\begin{equation*}
k(x,y)=\inf_\gamma\int_\gamma\frac{1}{\delta_\eps(\gamma(t))}\,ds_\eps(t),
\end{equation*}
where we took $\gamma$ to be arc-length parametrized with respect to the metric $d$ on $X$ with end points $x$ and $y$,
and $ds_\eps$ is the arc-length measure with arc-lengths computed with
respect to the uniformized metric $\de$.
By the construction of uniformization, we have that $ds_\eps(z)=e^{-\eps d(z,x_0)}\, ds$. On the other hand, by a straightforward calculation (see
Lemma~\ref{lem:dist-eps-bdy} below),
we know that $\delta_\eps(z)\simeq e^{-\eps d(z,x_0)}$. It follows that
\[
k(x,y)\simeq \inf_\gamma \ell(\gamma)=d(x,y). \qedhere
\]
\end{proof}

\begin{rmk}
The flip side of the above lemma is the following question. Suppose that $Z$ is a uniform space and
$X=(Z,k)$ the metric space obtained by considering the quasihyperbolic metric on $Z$. From~\cite{BHK} we
know that $X$ is then Gromov hyperbolic. Is there a choice of $\eps>0$ such that
$X_\eps$ is biLipschitz equivalent to $Z$? We do not know at this time whether such a choice of 
$\eps$ always exists. The difficulty underlying this question stems from the problem that 
the uniformization of two Gromov hyperbolic spaces that are biLipschitz equivalent
need not result in two biLipschitz equivalent metric spaces; uniformization is a more complex process
than quasihyperbolization. 

On the other hand, from~\cite[Corollary~1]{GeO}, 
$\Omega$ is a uniform domain if and only if the quasihyperbolic metric $k$ is equivalent
to the metric given by
\[
 j(x,y):=\log\left(1+\frac{d(x,y)}{\delta_Z(x)\wedge\delta_Z(y)}\right).
\]
Using the metric $j$ rather than $k$ to perform the uniformization procedure does result
in biLipschitz equivalence.
\end{rmk}

\section{Results}

Recall that $X$ and $Y$ are Gromov hyperbolic spaces and $\Phi:Y\to X$ is a $\tau$-rough 
isometry with a $\tau$-rough isometric inverse (in the sense of 
Remark~\ref{rmk:roughed up})
such that $\Phi(y_0)=x_0$ with $x_0\in X$ and $y_0\in Y$.
In what follows, all curves are assumed to be parametrized by (hyperbolic) arclength unless otherwise specified.

\begin{lemma}\label{discrete paths}
Suppose that $\rho \colon Y \to (0, \infty)$ 
satisfies the Harnack condition~\eqref{ratio comp} with constant $A$.
Let $L > 1$ and $\gamma \colon [0,L] \to Y$ be a curve with $\ell(\gamma)=L$.
Choose $N \in \N$ such that $N \leq L < N+1$. Then
\begin{equation}\label{eq:1}
\int_\gamma \rho ds \simeq \sum_{i=0}^{N-1} \rho(a_i),
\end{equation}
where $a_i=\gamma(iq)$ with $q:=\tfrac{L}{N}$. The comparison constant in~\eqref{eq:1}
can be taken to be $2A^2$.

If $L \leq Q$ with $Q\ge 1$ a fixed number,
we instead have $\int_\gamma \rho ds \simeq L \cdot \rho(\gamma(0))$ with comparison constant 
$A^{Q+1}$.
\end{lemma}

\begin{proof}
The statement dealing with the case $L\le Q$ is clear as $\rho$ satisfies the
Harnack condition; hence, we will only
consider the case $L\ge 1$.

Note that $1\le q<2$.  For $0 \le i \le N-1$, let 
$\gamma_i \colon [0,q] \to Y$ be the curve given by $\gamma_i(t) = \gamma(iq+t)$.  
Note that $\gamma_i$ is parametrized by arclength because $\gamma$ is.  Hence the length 
$\ell(\gamma_i)$ of $\gamma_i$ satisfies $1\le \ell(\gamma_i)<2$.
By condition~\eqref{ratio comp}, it follows that
\[
\frac{1}{A^2} \rho(a_i) \leq \int_{\gamma_i} \rho ds \leq 2A^2 \rho(a_i).
\]
Hence
\[
\sum_{i=0}^{N-1} \rho(a_i) \simeq \sum_{i=0}^{N-1} \int_{\gamma_i} \rho ds = \int_\gamma \rho ds
\]
with comparison constant $2A^2$.
\end{proof}

\begin{rmk}\label{path comp rmk}
 Lemma~\ref{discrete paths} holds in $X$ as well.
\end{rmk}

\begin{lemma}\label{Phi comp}
Suppose $x, y \in Y$ with $d(x,y) > 1$.  Let $L > 1$ and $\gamma \colon [0,L] \to Y$ be a curve with 
$\gamma(0)=x$ and $\gamma(L)=y$.  Fix $N \in \N$ such that $N \leq L < N+1$.    
Then,
\[
\int_\gamma \rey ds \simeq \sum_{i=0}^{N-1} \rex(\Phi(a_i)) \simeq \biggl( \sum_{i = 0}^{N-2} \rex(\Phi(a_i)) \biggr)+ \rex(\Phi(y))
\]
where $q = \tfrac{L}{N}$ and $a_i = \gamma(iq)$ for $0 \le i\le N$.
In the above, we adopt the convention that $\sum_{i=0}^{N-2}\rex(\Phi(a_0))=0$ if
$N = 1$.  The comparison constants depend solely on $\eps$ and $\tau$.
\end{lemma}

\begin{proof}
Note that $a_0=x$ and $a_N=y$.
For $0 \le i \le N$ let $b_i = \Phi(a_i)$.  By Lemma~\ref{discrete paths}, 
\[
\int_{\gamma} \rey ds \simeq \sum_{i=0}^{N-1} \rey (a_i)
\]
with comparison constant $2e^{2\eps}$.
Now, $\rey(a_i) = e^{-\eps d(y_0, a_i)}$ and, as $\Phi$ is a $\tau$-rough isometry, we have 
\[
d(y_0, a_i) - \tau \leq d(x_0, b_i) \leq d(y_0, a_i) + \tau.
\]
In particular, 
\[
e^{-\tau\eps} \leq \frac{\rey(a_i)}{\rex(b_i)} \leq e^{\tau\eps}
\]
for all $i$. Hence we have
\[
 \sum_{i=0}^{N-1} \rey (a_i)\simeq\sum_{i=0}^{N-1} \rex(\Phi(a_i)), 
\]
with comparison constant $e^{\tau\eps}$. Hence
\[
\int_\gamma \rey ds \simeq \sum_{i=0}^{N-1} \rex(\Phi(a_i))
\]
with comparison constant $2e^{2\eps+\tau\eps}$.

The second comparability follows as $d(a_{N-1},y) \leq 2$, and 
so $\rey(a_{N-1}) \simeq \rey(y)$ 
with comparison constant $e^{2\eps}$, see Remark~\ref{mult dist rmk}. 
\end{proof}

\begin{lemma}\label{lem:dist-eps-bdy}
Let $Y$ be a Gromov hyperbolic space and $\eps>0$. Then for each $x\in Y$
we have
\begin{equation*} 
\delta_\eps(x):=\dist(x,\partial Y_\eps):=\dist(x,\overline{Y_\eps}\setminus Y_\eps)
\gtrsim e^{-\eps d(x,y_0)}=\rey(x),
\end{equation*}
with comparison constant $1/\eps$. If in addition $Y$ is 
an $M$-roughly starlike space, then 
\[
\delta_\eps(x)\simeq\rey(x)
\]
with comparison constant $[M+\eps^{-1}]e^{\eps M}$.
\end{lemma}

\begin{proof}
Let $x\in Y$ and $\gamma$ be any path from $x$ that leaves every compact subset of $Y$.
Then we have
\[
\int_\gamma e^{-\eps d(\gamma(t),y_0)}\, dt\ge \int_0^\infty e^{-\eps[d(y_0,x)+t]}\, dt
=\frac{e^{-\eps d(y_0,x)}}{\eps}.
\]
Taking the infimum over all such $\gamma$ gives
\[
\delta_\eps(x)\ge \frac{\rey(x)}{\eps}.
\]

Now suppose that $Y$ is also $M$-roughly starlike. 
Let $x\in Y$ and 
$\gamma:[0,\infty)\to Y$ be a geodesic ray from $y_0$ so that there is some 
$t_0\in[0,\infty)$ for which 
$d(x,\gamma(t_0))\le M$. Let $\beta$ be a geodesic with end points $x$ and $\gamma(t_0)$; then
the concatenation $\gamma_*$ of $\gamma\vert_{[t_0,\infty)}$ and $\beta$ gives us that
\[
\delta_\eps(x)\le \int_{\gamma_*}e^{-\eps d(\gamma_*(t),y_0)}\, dt.
\]
Note that for points $w\in\beta$, $d(x,y_0)-M\le d(w,y_0)\le d(x,y_0)+M$, and so
\[
\delta_\eps(x)\le Me^{\eps M}e^{-\eps d(x,y_0)}+\int_{t_0}^\infty e^{-\eps t}\, dt
  \le Me^{\eps M}e^{-\eps d(x,y_0)}+\eps ^{-1}e^{-\eps t_0}.
\]
Moreover, $t_0=d(\gamma(t_0),y_0)\ge d(y_0,x)-M$. Therefore
\[
\delta_\eps(x)\le [M+\eps^{-1}]e^{\eps M}\, e^{-\eps d(y_0,x)}.
\]
It follows that
\begin{equation*} 
\delta_\eps(z)\simeq e^{-\eps d(z,y_0)},
\end{equation*}
with comparison constant $[M+\eps^{-1}]e^{\eps M}$.
\end{proof}

In the proof of the following lemma we use $\Phi^{-1}$ together with $\Phi$, see 
Lemma~\ref{lem:inverse-rough-isom} regarding the construction of
$\Phi^{-1}$.

\begin{lemma}\label{lem:compare-deltaSubEps}
Let $X$ and $Y$ be two Gromov hyperbolic spaces and $\Phi:Y\to X$ be a $\tau$-rough isometry. 
Then for $\eps>0$ and for each $y\in Y$,
\[
\delta_\eps(y)\simeq\delta_\eps(\Phi(y))
\]
with the comparison constant depending solely on $\eps$ and $\tau$.
\end{lemma}

\begin{proof}
Let $y\in Y$ and $x:=\Phi(y)$. Let $\gamma:[0,\infty)\to Y$ 
be any path from $y$ that leaves every compact subset of $Y$. Set $a_0:=y$ and
for $k\in\N$ let $a_k:=\gamma((1+\tau)k)$. 
Recall the definition of $\ell_\eps(\gamma)$ from Definition~\ref{def:uniformization}.
Then a simple modification of Lemma~\ref{discrete paths} together with the rough isometric 
equivalence of $X$ and $Y$ tells us that 
\[
 \ell_\eps(\gamma)\simeq\sum_{k=0}^\infty \rey(a_k)\simeq\sum_{k=0}^\infty\rex(\Phi(a_k)).
\]
Let $\beta_k$ be a hyperbolic geodesic in $X$ with end points $\Phi(a_k)$ and $\Phi(a_{k+1})$
and $\beta$ be the concatenation of the paths $\beta_k$, $k=0,1,\cdots$.
Since $\ell(\gamma\vert_{[k,k+1]})=1+\tau$, it follows that 
$1\le d(\Phi(a_k),\Phi(a_{k+1}))\le 1+2\tau$;
therefore $1\le \ell(\beta_k)\le 1+2\tau$. 
Therefore by the second part of Lemma~\ref{discrete paths} and the above estimate,
\[
\ell_\eps(\beta)=\sum_{k=0}^\infty\ell_\eps(\beta_k)
\simeq\sum_{k=0}^\infty\rex(\Phi(a_k))\simeq\ell_\eps(\gamma).
\]
Note that as $\Phi$ is a rough isometry and $\gamma$ leaves every bounded subset of the
proper space $Y$, the path $\beta$ also leaves every bounded subset of $X$.
Since we require in this paper that $X$ and $Y$ are locally compact geodesic spaces,
we know from the Hopf-Rinow theorem that $X_\eps$ and $Y_\eps$ are length spaces.
Hence, taking the infimum over all $\gamma:[0,\infty)\to Y$ with $\gamma(0)=y$ as above yields
\[
\ell_\eps(\beta)\lesssim \delta_\eps(y).
\]
It follows that
\[
\delta_\eps(\Phi(y))\le \ell_\eps(\beta) \lesssim\delta_\eps(y).
\]
Reversing the roles of $X$ and $Y$, and replacing $\Phi$ with $\Phi^{-1}$ gives
\[
 \delta_\eps(\Phi^{-1}(\Phi(y)))\lesssim\delta_\eps(\Phi(y)).
\]
Since $d(y,\Phi^{-1}(\Phi(y)))\le \tau$, it follows from Lemma~\ref{lem:dist-eps-bdy} 
and the second part of Lemma~\ref{discrete paths} that
\begin{align}\label{eq:dist-to-bdy-closeby}
\delta_\eps(y)\le \delta_\eps(\Phi^{-1}(\Phi(y)))+d_\eps(y,\Phi^{-1}(\Phi(y)))
 &\lesssim \delta_\eps(\Phi^{-1}(\Phi(y)))+\tau\rey(\Phi^{-1}(\Phi(y)))\notag\\
\lesssim\delta_\eps(\Phi^{-1}(\Phi(y))).
\end{align}
\end{proof}

\begin{lemma} \label{lem:uniformizing-close-by-points}
Let $x,y\in Y$ be such that $d(x,y)\le 4+\tau$, and let $\gamma$ be a 
hyperbolic geodesic in $Y$ with
end points $x,y$. Then 
\begin{equation}\label{eq:d-eps-est-4-close-points}
\ell_\eps(\gamma)\simeq d_\eps(x,y)\simeq e^{-\eps d(x,y_0)}d(x,y)
\end{equation}
and $\gamma$ is a uniform curve with respect to the metric $\de$ on $Y_\eps$, with
uniformity constant depending only on $\eps$, and $\tau$.
\end{lemma}

\begin{proof}
Recall from Definition~\ref{def:uniformization}
that the length $\ell_\eps(\gamma)$ of $\gamma$ in the uniformized metric $\de$ is given by
\[
\ell_\eps(\gamma)=\int_\gamma e^{-\eps d(\gamma(t),y_0)}\, dt,
\]
and as 
\[
d(x, y_0) - 4 - \tau \leq d(x,y_0)-d(x,z)\le d(y_0,z)\le d(x,y_0)+d(x,z) \leq d(x, y_0) + 4 + \tau.
\]
 for each $z$ in the trajectory of $\gamma$, we see that
\[
\ell(\gamma) e^{-\eps d(x,y_0)}e^{-\eps(4+\tau)} \leq \ell_\eps(\gamma) \leq \ell(\gamma) e^{-\eps d(x, y_0)}e^{\eps(4+\tau)}.
\]
Observe that $d(x,y)=\ell(\gamma)$.
On the other hand, with $\beta$ any rectifiable non-geodesic curve in $Y$ with end points $x$ and $y$, 
we must have $\ell(\beta)>d(x,y)$, and so with $t_0\in[0,\ell(\beta)]$ the smallest number for which
$d(x,\beta(t_0))=d(x,y)$, we get
\[
\int_\beta\rey\, ds \geq \int_0^{t_0}\rey\circ\beta(t)\, dt\ge  d(x,y) e^{-\eps d(x,y_0)} e^{-\eps (4 + \tau)}.
\]
Therefore 
\[
d(x,y)e^{-\eps d(x, y_0)}e^{\eps(4+\tau)} \ge 
\ell_\eps(\gamma)\ge \de(x,y)=\inf_\beta\int_\beta\rey\, ds\ge d(x,y) e^{-\eps d(x,y_0)} e^{-\eps (4 + \tau)}.
\]
It immediately follows that
\[
\de(x,y)\simeq d(x,y) e^{-\eps d(x,y_0)},
\]
and consequently, we have that $\ell_\eps(\gamma)\lesssim \de(x,y)$,
that is, $\gamma$ is a $C$-quasiconvex curve in $Y_\eps$, with constant $C$ depending only on $\eps$ and $\tau$. 
Recall that a curve is a $C$-quasiconvex if its length is dominated by at most $C$ times the distance between
the endpoints of the curve (see Definition~\ref{def:uniform}). Moreover, 
from Lemma~\ref{lem:dist-eps-bdy} and the fact that $d(x,y)\le 4+\tau$ 
we know that for $z\in\gamma$,
\[
\delta_\eps(z)\gtrsim e^{-\eps d(z,y_0)}\gtrsim e^{-\eps d(x,y_0)}\gtrsim\ell_\eps(\gamma),
\]
that is, $\gamma$ is a uniform curve, with a uniformity constant that depend only on $\eps$
and $\tau$.
\end{proof}

From the above lemma, to show that $Y_\eps$ is a uniform domain it suffices to show that $x,y\in Y$ can be 
connected by a uniform curve when $d(x,y)\ge 4+\tau$. This is the focus of the remaining discussion.

\begin{lemma}\label{de comp}
Let $x,y \in Y$ be such that $d(x,y) \ge 2 + \tau$.  Then
\[
\de(x,y) \simeq \de (\Phi(x), \Phi(y)).
\]
\end{lemma}

See Lemma~\ref{lem:inverse-rough-isom} regarding the construction of
$\Phi^{-1}$.

\begin{proof}
Let $\gamma \colon [0, L] \to Y$ be any curve with $\gamma(0) = x$, $\ell(\gamma)=L$, and $\gamma(L) = y$.  Note that
$L \ge 2+\tau \ge 2$.  We fix $N \in \N$ such that $N \le L < N+1$.  
Let $q = \tfrac{L}{N}$ and, for $0 \le i \le N$, let $a_i = \gamma(i q)$ and $b_i = \Phi(a_i)$.  
Then $d(b_i,b_{i+1}) \leq d(a_i,a_{i+1})+\tau \le 4 + \tau$,
and so by Lemma~\ref{lem:uniformizing-close-by-points} we have 
\[
\de(b_i, b_{i+1}) \lesssim e^{-\eps d(b_i, x_0)}  = \rex(b_i)
\]
with comparability constant depending only on $\eps$ and $\tau$. 
It follows that
\[
\de(\Phi(x), \Phi(y)) \leq \sum_{i=0}^{N-1} \de(b_i, b_{i+1}) \lesssim \sum_{i=0}^{N-1} \rex(b_i).
\]
By Lemma~\ref{Phi comp}, we have $\sum_{i=0}^{N-1} \rex(b_i) \simeq \int_{\gamma} \rey ds$.  
Infimizing over all paths $\gamma$ connecting $x$ to $y$ yields
\[
\de(\Phi(x), \Phi(y)) \lesssim \inf_{\gamma} \int_{\gamma} \rey ds = \de(x,y) .
\]

Next, note that $d(\Phi(x),\Phi(y)) \ge d(x,y)-\tau \geq 2$.  Hence, for $\Phi^{-1}(\Phi(x)) = x'$ and 
$\Phi^{-1}(\Phi(y)) = y'$ we can apply the same argument above to conclude that
\[
\de(x', y') \lesssim \de(\Phi(x), \Phi(y)).
\]

It remains to relate $\de(x', y')$ with $\de(x,y)$.  
As $d(\Phi^{-1}\circ\Phi(z),z)\le \tau$ for each $z\in Y$, it follows from   
Lemma~\ref{lem:uniformizing-close-by-points} that 
$\de(x',x) \lesssim e^{-\eps d(x,y_0)}$ and $\de(y',y) \lesssim e^{-\eps d(y,y_0)}$.
Moreover, $d(x,y_0)\ge d(\Phi(x),x_0)-\tau$ and $d(y,y_0)\ge d(\Phi(y),x_0)-\tau$.
Hence,
\[
\de(x,y)\le \de(x,x')+\de(x',y')+\de(y',y)\lesssim \de(\Phi(x),\Phi(y))+e^{-\eps d(\Phi(x),x_0)}+e^{-\eps d(\Phi(y),x_0)}.
\]
Since $d(\Phi(x),\Phi(y))\ge d(x,y)-\tau\ge 2$, we apply  
Lemma~\ref{discrete paths} to see that
whenever $\beta$ is a rectifiable curve in $X$ with end points $\Phi(x)$ and $\Phi(y)$,
\[
\int_\beta \rho_\eps^X\, ds\simeq \sum_{i=0}^{N-1}\rho_\eps^X(a_i)\gtrsim \rho_\eps^X(\Phi(x))+\rho_\eps^X(\Phi(y)).
\]
Taking the infimum over all such $\beta$ gives
\[
\de(\Phi(x),\Phi(y))\gtrsim e^{-\eps d(\Phi(x),x_0)}+e^{-\eps d(\Phi(y),x_0)},
\]
from which we obtain the desired conclusion
\[
\de(x,y)\lesssim \de(\Phi(x),\Phi(y)). \qedhere
\] 
\end{proof}

\begin{thm}\label{thm:main}
Let $(X,d)$ and $(Y,d)$ be two complete Gromov hyperbolic geodesic spaces, and suppose that there exists
a $\tau$-rough isometry $\Phi:Y\to X$. Let $y_0\in Y$ and set $x_0=\Phi(y_0)$. If 
$\eps>0$ is such that $(X_\eps,\de)$ is a uniform domain with a uniformity constant $\lambda$, then
$(Y_\eps,\de)$ is also a uniform domain with a uniformity constant that depends solely on $\lambda$, $\eps$,
and $\tau$.
\end{thm}

\begin{proof}
Let $x,y\in Y$. If $d(x,y)\le 4+\tau$, then by Lemma~\ref{lem:uniformizing-close-by-points} we know that
the hyperbolic geodesic connecting $x$ to $y$ is a uniform curve in $(Y_\eps,\de)$. Therefore to verify that
$Y_\eps$ is a uniform domain, it suffices to consider only points $x,y\in Y$ with $d(x,y)>4+\tau$. For such $x,y$
we have that $d(\Phi(x),\Phi(y))\ge 4$. Let $\gamma$ be a uniform curve in $X_\eps$ with end points
$\Phi(x), \Phi(y)$. Then $\ell(\gamma)\ge 4$, and so we can apply Lemma~\ref{discrete paths} to $\gamma$.
With $a_i=\gamma(iq)$, $q=L/N$, we see that 
\[
\de(x,y)\simeq \de(\Phi(x),\Phi(y))\simeq \int_\gamma\rex\, ds.
\]
Here we have also used Lemma~\ref{de comp}.
Applying Lemma~\ref{Phi comp} with $\Phi^{-1}:X\to Y$ playing the role of $\Phi$ there, we obtain
\[
\de(x,y)\simeq \sum_{i=0}^{N-2}\rey(\Phi^{-1}(a_i))+\rey(\Phi^{-1}(\Phi(y))).
\]
As $d(y,\Phi^{-1}\circ\Phi(y))\le \tau$ and $d(x,\Phi^{-1}\circ\Phi(x))\le \tau$, we have that
\[
\de(x,y)\simeq \rey(x)+\rey(y)+\sum_{i=1}^{N-2}\rey(\Phi^{-1}(a_i)).
\]
Note that $d(a_i,a_{i+1})\le 2$, and so $d(\Phi^{-1}(a_i),\Phi^{-1}(a_{i+1}))\le 2+\tau$. 
Similarly, $d(x,\Phi^{-1}(a_1))\le 2+2\tau$, $d(y,\Phi^{-1}(a_{N-1}))\le 2+2\tau$. 
We set $\beta_0$ to be the hyperbolic geodesic with end points $\Phi^{-1}(a_1)$ and $x$,
and set $\beta_{N-1}$ to be the hyperbolic geodesic with end points $y$ and $\Phi^{-1}(a_{N-1})$.
For $i=1,\cdots, N-2$ let $\beta_i$ be 
the hyperbolic geodesic in $Y$ with end points $\Phi^{-1}(a_i)$ and $\Phi^{-1}(a_{i+1})$. By 
Lemma~\ref{lem:uniformizing-close-by-points} 
we have that
\[
\ell_\eps(\beta_i)\simeq \rey(\Phi^{-1}(a_i))d(\Phi^{-1}(a_i),\Phi^{-1}(a_{i+1}))\lesssim \rey(\Phi^{-1}(a_i)), 
\]
and so
\[
\de(x,y)\gtrsim \sum_{i=0}^{N-1}\ell_\eps(\beta_i)=\ell_\eps(\beta),
\]
where $\beta$ is the concatenation of the finitely many curves $\beta_i$, $i=0,\cdots, N-1$.
Thus $\beta$ is a quasiconvex curve connecting $x$ to $y$ in $Y$. We now show that this curve
is a uniform curve, that is, it satisfies Condition~2 of Definition~\ref{def:uniform}.

Let $z\in\beta$. If $z\in\beta_0\cup\beta_{N-1}$, then the result follows from Lemma~\ref{lem:uniformizing-close-by-points}.
Thus we may assume that $z\in\beta_i$ for some $i\in\{1,\cdots, N-2\}$. Then 
by~\eqref{eq:dist-to-bdy-closeby}, Lemma~\ref{lem:compare-deltaSubEps}, and by the
uniformity of $\gamma$, we have 
\[
\delta_\eps(z)\simeq\delta_\eps(\Phi^{-1}(a_i))\simeq\delta_\eps(a_i)
  \ge\frac{1}{\lambda}\ell_\eps(\gamma[\Phi(x),a_i]).
\]
Here we assume that $\ell_\eps(\gamma[\Phi(x),a_i])=\min\{\ell_\eps(\gamma[\Phi(x),a_i]),
\ell_\eps(\gamma[\Phi(y),a_i])\}$, since if this is not the case, we reverse the roles of
$x$ and $y$ (and sum over all $j$ from $i$ to $N$) in the following estimates.
A repeat of the arguments above also tell us that
\[
\ell_\eps(\gamma[\Phi(x),a_i])\simeq \sum_{j=0}^i\rex(a_j)\simeq\sum_{j=0}^i\rey(\Phi^{-1}(a_j))
\gtrsim \ell_\eps(\beta[x,z]).
\]
Combining the above estimates, we obtain $\delta_\eps(z)\gtrsim\ell_\eps(\beta[x,z])$.
\end{proof}

\section{On $(L,C)$-rough similarities}

A natural extension of the notion of rough isometries is the notion of rough similarities, which is a 
proper subclass of a wider class of mappings known as quasiisometries. 
A map $\Phi:Y\to X$ is said to be an $(L,C)$-rough similarity if $L>0$
and $C\ge 0$ are such that for every pair $x,y\in Y$ we have
\[
Ld_Y(x,y)-C\le d_X(\Phi(x),\Phi(y))\le Ld_Y(x,y)+C
\] 
and the Hausdorff distance between $\Phi(Y)$ and $X$ is at most $C$.
Thus a natural question to ask in the setting considered in this paper is that if both $X$ and $Y$ are
Gromov hyperbolic in the sense of Definition~\ref{def:GromovHyp}, 
$\eps>0$ is
such that $(X_\eps, d_\eps)$ is a uniform domain, and $\Phi:Y\to X$ is an $(L,C)$-rough
similarity,
then is $(Y_\eps, d_\eps)$ also a uniform domain? Our main theorem, Theorem~\ref{thm:main},
answers this in the affirmative when $L=1$. Unfortunately this theorem fails for more general $L$, as
we will see in this section.

Let $Y$ be the complex unit disk, equipped with the hyperbolic metric. Then, for $R>0$, the 
(hyperbolic) circle $C_R$ centered at $0$ with radius $R$ has (hyperbolic) length
\[
\ell_Y(C_R)=\tfrac{\pi}{2}(e^{2R}-e^{-2R}). 
\]
We now fix $\eps>2$.  
Observe from~\cite{BHK} that there is some $\eps_1>0$ such that $Y_{\eps_1}$ is a uniform domain.
Let $Z$ be the same complex unit disk, but equipped with the scaled hyperbolic metric given by 
$d_Z(z,w)=\tfrac{\eps_1}{\eps} d(z,w)$. 
We know that $Z_\eps=Y_{\eps_1}$, and so $Z_\eps$ is a uniform domain. Clearly $Z$ is 
$(\tfrac{\eps_1}{\eps},0)$-roughly similar to $Y$.
We now
show that $Y_\eps$ is \emph{not} a uniform domain. For $R\gg1$ let 
\[
z_R=\left(\tfrac{e^{2R}-1}{e^{2R}+1}, 0\right) \text{ and }
w_R=\left(-\tfrac{e^{2R}-1}{e^{2R}+1}, 0\right).
\]
The length of the hyperbolic geodesic ray 
$[z_R,(1,0))\subset\R$
from $z_R$ to $\infty$ has length $\tfrac{1}{\eps} e^{-\eps R}$,
and, as $\eps>2$, we have that $\partial Y_\eps$ has only one point.
Therefore,
\[
d_\eps(z_R,w_R)\le \tfrac{2}{\eps}e^{-\eps R}. 
\]
Suppose that $Y_\eps$ is a uniform domain with
a uniformity constant $A$. Then let $\gamma$ be an $A$-uniform curve with end points $z_R$ and $w_R$. 
Since $\gamma$ is a compact curve, it stays within a closed hyperbolic disk centered at $0$. Let $T$ be the
smallest radius of such a hyperbolic disk; then there is some point $z_0$ in $\gamma$ with $d_Y(z_0,0)=T$.
By symmetry, we may assume that $\ell_\eps(\gamma_{z_R,z_0})\le \ell_\eps(\gamma_{w_R,z_0})$. Then
\[
\ell_\eps(\gamma_{z_R,z_0})\ge \int_R^T e^{-\eps t}\, dt=\frac{e^{-\eps R}-e^{-\eps T}}{\eps}.
\] 
Thus the second condition of uniformity of $\gamma$ from Definition~\ref{def:uniform} implies that
\[
T\le \frac{\log(A+1)}{\eps}+R.
\]
Let $\tau\ge 0$ be the largest number such that $\gamma$ lies in the complement of the hyperbolic
ball centered at $0$ with radius $\tau$ (with that ball being empty if $\tau=0$). As before, it then follows that
\[
\ell_\eps(\gamma)\ge \frac{e^{-\eps \tau}-e^{-\eps R}}{\eps},
\]
which, when combined together with the first condition of uniformity of $\gamma$ from Definition~\ref{def:uniform},
tells us that
\[
 \frac{e^{-\eps \tau}-e^{-\eps R}}{\eps}\le A d_\eps(z_R,w_R)\le \frac{2A}{\eps} e^{-\eps R}.
\]
From this we also conclude that necessarily
\[
\tau\ge R-\frac{\log(2A+1)}{\eps}.
\]
Thus the curve $\gamma$ is trapped in the annulus, centered at $0$, with (hyperbolic) inner radius 
$R-\tfrac{\log(2A+1)}{\eps}$ and (hyperbolic) outer radius $R+\tfrac{\log(A+1)}{\eps}$.
It then follows that (with $\gamma$ parametrized with respect to the arc length measured from
the point of view of the hyperbolic metric)
\[
\ell_\eps(\gamma)=\int_\gamma e^{-\eps d_Y(\gamma(t),0)}\, ds(t)\ge e^{-\eps [R+\log(A+1)/\eps]} \ell_Y(\gamma)
 \ge (A+1)e^{-\eps R} \, \frac{\pi}{8} e^{2[R-\log(2A+1)/\eps]},
\]
where we used the fact that $R\gg 1$. Now applying the first condition of uniformity of $\gamma$ again tells 
us that with $C_A=(A+1)\tfrac{\pi}{8} (2A+1)^{-2/\eps}$,
\[
C_A e^{-(\eps -2)R}\le A d_\eps(z_R,w_R)\le \frac{2A}{\eps} e^{-\eps R}.
\]
As this is not possible for large $R$, it follows that $Y_\eps$ is not uniform with uniformity constant $A$. Since
the constant $A$ above is arbitrary, it follows that $Y_\eps$ is \emph{not} a uniform space.

The above computation also indicates that if $Y$ is a Gromov hyperbolic space, $C>0$ and $y_0\in Y$ such that
for $R\gg H\gg 1$ there are points $z_R,w_R\in Y$ with $d(z_R,y_0)=d(w_R,y_0)$ so that
whenever $\gamma\subset B(y_0,R+H)\setminus B(y_0,R-H)$ has $z_R$ and $w_R$ as its end points we must
have $\ell(\gamma)\ge c_H e^{C R}$, then $Y_\eps$ cannot be a uniform space when $\eps >C$. Should
$Y$ be a smooth hyperbolic 
manifold, the smallest choice of such $C$ is associated with the curvature of $Y$; however,
as we are merely concerned with the large scale behavior here, this association is not straight-forward. In the
setting of (non-smooth) metric spaces a similar limitation on $\eps$ in terms of some notion of curvature might
be possible. Indeed, from the work of~\cite{BHK} we know that for each $\delta>0$
there is some $\eps_0(\delta)>0$ such that 
if $Y$ is Gromov hyperbolic with hyperbolicity constant
$\delta$ as in Definition~\ref{def:GromovHyp}, 
then $Y_\eps$ is a uniform domain whenever $\eps\le \eps_0(\delta)$. The estimates for $\eps_0(\delta)$ are
not explicit in~\cite{BHK}, and may perhaps not even be optimal. However, the results of~\cite{BHK} indicate a link between 
some notion of
curvature of $Y$ and the supremum of all $\eps>0$ for which $Y_\eps$ is a uniform domain.
As this link is not well understood so far, and the estimates for the limitation $\eps_0(\delta)$ in~\cite{BHK}
are not known to be optimal, Theorem~\ref{thm:main} gives a way of verifying uniformity of $Y_\eps$ without 
linking it to $\delta$ but instead comparing it to all other locally compact geodesic metric spaces that are
roughly isometric to $Y$.

\vskip .2cm

\noindent Address: University of Cincinnati, Department of Mathematical Sciences, P.O. Box 210025, 
Cincinnati, OH 45221-0025, U.S.A.

\noindent E-mail: 

J.L.: {\tt lindqujy@ucmail.uc.edu}

N.S.: {\tt shanmun@uc.edu}

\end{document}